\documentclass[12pt, reqno]{amsart}
\usepackage{amsmath, amsthm, amscd, amsfonts, amssymb}
\usepackage{bm}

\newtheorem{theorem}{Theorem}[section]
\newtheorem{lemma}[theorem]{Lemma}

\newtheorem{corollary}[theorem]{Corollary}
\theoremstyle{definition}

\numberwithin{equation}{section}

\newcommand{\vp}{\varphi}

\newcommand{\clb}{\mathcal{B}}

\newcommand{\D}{\mathbb{D}}

\newcommand{\T}{\mathbb{T}}
\newcommand{\Z}{\mathbb{Z}}

\newcommand{\raro}{\rightarrow}

\newcommand{\C}{\mathbb{C}}
\newcommand{\M}{M}
\newcommand{\he}{H^2_E(\D)}
\newcommand{\Le}{L^2_E(\T)}
\newcommand{\linf}{L^{\infty}_{\clb(E)}(\T)}
\newcommand{\hinf}{H^{\infty}_{\clb(E)}(\D)}
\newcommand{\zhinf}{\overline{zH^{\infty}_{\clb(E)}(\D)}}
\begin{document}

%\today

\setcounter{page}{1}

%\today

\title[Similarity to contractions]{Foguel-type operators similar to contractions}

\author[Das]{Nilanjan Das}
\address{Indian Statistical Institute, Statistics and Mathematics Unit, 8th Mile, Mysore Road, Bangalore, 560059,
India}
\email{nilanjand7@gmail.com}

\author[Das]{Soma Das}
\address{Indian Statistical Institute, Statistics and Mathematics Unit, 8th Mile, Mysore Road, Bangalore, 560059, India}
\email{dsoma994@gmail.com}

\author[Sarkar]{Jaydeb Sarkar}
\address{Indian Statistical Institute, Statistics and Mathematics Unit, 8th Mile, Mysore Road, Bangalore, 560059,
India}
\email{jay@isibang.ac.in, jaydeb@gmail.com}

%\today



\begin{abstract}
Pisier's celebrated counterexample to Halmos's similarity-to-contractions problem was based on $2 \times 2$ upper triangular block operator matrices involving three classical operators: forward and backward shifts on the diagonal and Hankel operators in the off-diagonal entry. Together with another classical object, namely Toeplitz operators, one can formulate another $2^3 -1 = 7$ types of $2 \times 2$ upper triangular block operator matrices, which we refer to as Foguel-type operators. In this paper, we give a complete characterization of all the seven Foguel-type operators being similar to contractions.
\end{abstract}

\maketitle

\tableofcontents

\section{Introduction}\label{sec intro}

In this paper, $E$ will be treated as an arbitrary but fixed separable Hilbert space over $\C$, and $\clb(E)$ will denote the space of all bounded linear operators $T: E \raro E$ with norm
\[
\|T\|: = \sup \{\|Tf\|: \|f\| = 1, f \in E\}.
\]
An operator $T \in \clb(E)$ is said to be a \textit{contraction} if $\|T\| \leq 1$. Equivalently,
\[
I - T^*T \geq 0.
\]
This positivity property plays a crucial role in analyzing the structure of linear operators. Evidently, not all operators enjoy such positivity properties. A natural question therefore arises: can a non-contractive operator be made into a contractive one by using an equivalent norm on $E$? This question is equivalent to the celebrated \textit{similarity problem}, which specifically asks for a characterization of when a given operator is similar to a contraction. Recall that $T \in \clb(E)$ is said to be \textit{similar to a contraction} if there exists an invertible operator $X \in \clb(E)$ such that
\[
X^{-1} T X,
\]
is a contraction. This is a classical problem, with its roots in Sz.-Nagy's 1947 paper on the characterization of operators similar to unitary operators \cite{Nagy}: $T \in\clb (E)$ is similar to a unitary operator if and only if
\[
\sup_{n\in\Z}\|T^n\|<\infty.
\]
In pointing out this result as a starting point of the similarity problem, Pisier remarked \cite{Pisier1} that Sz.-Nagy was “almost surely motivated by his own work with F. Riesz in ergodic theory, where contractions play a key role.” In 1959, this result led Sz.-Nagy \cite{Nagy 59} to consider the class of \textit{power bounded} operators, that is, operators $T \in \clb(E)$ such that
\[
\sup_{n \in \Z_+} \|T^n\| < \infty,
\]
and to the natural question:
\[
\text{Are all power bounded operators similar to contractions?}
\]
In the same paper \cite{Nagy 59}, Sz.-Nagy gave a positive answer to this question for all compact operators. However, in 1964, Foguel \cite{Foguel} answered Sz.-Nagy's question in the negative in the general case (also see Halmos \cite{H-64}). It was further observed, in view of the von Neumann inequality \cite{vN}, that a natural replacement of ``power boundedness'' is ``polynomial boundedness'': $T\in \clb(E)$ is \textit{polynomially bounded} if there exists $C>0$ such that
\[
\|p(T)\|\leq C\sup_{z\in\overline{\D}}|p(z)|,
\]
for all complex polynomials $p$, where $\D$ is the open unit disc in $\C$. In his celebrated 1970 paper \cite{Halmos}, Halmos posed the following question:
\[
\mbox{Are all polynomially bounded operators similar to contractions?}
\]
About 25 years later, the Halmos conjecture was refuted by Pisier \cite{Pisier2}. In the meantime, Paulsen \cite{Paulsen1} established a characterization: An operator is similar to a contraction if and only if it is completely polynomially bounded. In this context, also see Aleksandrov and Peller \cite{Aleksandrov}, and Bourgain \cite{Burgain}. We also refer readers to the excellent survey articles \cite{Davidson, Pisier1} for a more comprehensive overview of the similarity problem.

Denote by $H^2_E(\D)$ the $E$-valued Hardy space over $\D$ \cite{Radj-Rosen}. For each $X \in \clb(\he)$, define the block operator matrix $M^{S_{E}^*}_{S_{E}}(X)$ acting on $H^2_{E}(\D) \oplus H^2_{E}(\D)$ by
\[
M^{S_{E}^*}_{S_{E}}(X): = \begin{bmatrix}
S_{E}^* & X\\
0 & S_{E}
\end{bmatrix},
\]
where $S_E$ denotes the shift operator on $H^2_E(\D)$ \cite{Radj-Rosen}. Well before Pisier’s counterexample, this particular class of operators (in the case of $E = \C$) was proposed by Foia\c{s} and Williams \cite{Foias-Williams} and Peller \cite{VPeller} as a potential candidate for the study of Halmos’s similarity problem. One motivation for studying this class of operators comes from Foguel’s work \cite{Foguel}. His counterexample to Sz.-Nagy’s question was an operator of the form
\[
M^{S^*}_S(X) := \begin{bmatrix}
S^* & X\\
0 & S
\end{bmatrix},
\]
acting on $H^2(\D) \oplus H^2(\D)$, where $H^2(\D)$ denotes the scalar-valued Hardy space (in this case, we omit $\mathbb{C}$ from all subscripts). We remark that the operator $X$ in Foguel's counterexample is a specific orthogonal projection associated with a lacunary sequence.

In his counterexample to Halmos's question, Pisier considered operators of the form $M^{S_{E}^*}_{S_{E}}(X)$, where
\[
\text{dim} E = \infty,
\]
and $X$ is a Hankel operator related to anti-commutation relations. His example (and related results) suggests that the similarity of $M^{S_{E}^*}_{S_{E}}(X)$ to a contraction, where $X$ is a Hankel operator, is a rather subtle matter. In fact, a clear-cut characterization of those Hankel operators $X$ for which $M^{S_{E}^*}_{S_{E}}(X)$ is similar to a contraction is still unknown.

Following the approach of Foguel and Pisier, this paper examines the similarities to contractions for \textit{Foguel-type operators} defined as
\[
M^Y_Z(X)= \begin{bmatrix}
Y & X\\
0 & Z
\end{bmatrix},
\]
acting on $H^2_E(\D) \oplus H^2_E(\D)$, whenever
\[
Y = S_E \quad \text{ and } \quad Z = S_E \text{ or } S^*_E,
\]
or
\[
Y = Z = S_E^*,
\]
along with
\[
X = \text{ Toeplitz or Hankel operator},
\]
including the case where $Y = S_E^*$, $Z = S_E$, and $X =$ Toeplitz operator. We present a complete characterization of Foguel-type operators that are similar to contractions. Therefore, together with the case resolved by Pisier’s counterexample, this completes the analysis of the similarity problem for all seven $2 \times 2$ upper triangular block operator matrices whose diagonal entries are forward and backward shift operators and whose off-diagonal entries are either Toeplitz or Hankel operators.

Before presenting the main similarity results obtained in this paper, we introduce some notations. By $\Le$, we denote the $E$-valued $L^2$-space on $\T$ ($=\partial \D$). Denote by $\linf$ the Banach space of all $\clb(E)$-valued essentially bounded measurable functions on $\T$. Also, let $\hinf$ denote the Banach space of $\clb(E)$-valued bounded analytic functions on $\D$. Let $\Phi \in \linf$. The \emph{multiplication operator} $M_\Phi \in \clb(L^2_E(\T))$ is defined by
\[
M_\Phi(f)=\Phi f,
\]
for all $f\in L^2_E(\T)$. The {\it Toeplitz} operator $T_\Phi\in \clb(\he)$ and the {\it Hankel} operator $H_\Phi\in \clb(\he)$ with symbol $\Phi$ are defined as
\[
T_\Phi=P_{\he}M_\Phi|_{\he}\,\,\mbox{and}\,\,H_\Phi=P_{\he}M_\Phi J|_{\he},
\]
respectively, where $J:\Le \to \Le$ is the {\it flip} operator, defined by $(Jf)(z) =f(\bar{z})$ for all $f\in \Le$ and $z\in \T$-a.e.

Given a closed subspace $L$ of a Hilbert space $H$, we denote by $P_L$ the orthogonal projection from $H$ onto $L$. In this context, note that in the discussion above, we have treated $\he$ as a closed subspace of $L^2_E(\T)$ \cite{Radj-Rosen}. In what follows, and whenever it is clear from the context, $E$ will be treated as a closed subspace of $\he$. For brevity, we will frequently use the notation $\Phi P_E$ to denote the operator $M_\Phi P_E$, that is,
\[
M_\Phi P_E = \Phi P_E,
\]
for each $\Phi\in L_{\clb(E)}^\infty(\T)$. Moreover, in order to conveniently state the results, we introduce the class $\zhinf$, which consists of $\Phi \in \linf$ such that
\[
\Phi = \sum_{n=1}^{\infty}\bar{z}^n\Phi_n,
\]
where $\Phi_n \in \clb(E)$ for all $n \geq 1$. The following is a summary of the main similarity results concerning Foguel-type operators obtained in this paper:

\begin{theorem}\label{thm: intr}
Let $\Phi \in \linf$. The following hold:
\begin{enumerate}
\item $\M^{S_E}_{S^*_E}(T_{\Phi})$ is similar to a contraction if and only if
\[
\Phi \in (1-\bar{z}^2)\linf,
\]
and
\[
\Phi^* + \widetilde{\Phi}^* \in (1-z^2)(\hinf +\zhinf),
\]
where $\widetilde{\Phi} (z) = \Phi(\bar{z})$, $z \in \T$ a.e.
\item $\M^{S_E}_{S^*_E}(H_{\Phi})$ is similar to a contraction if and only if
\[
H_\Phi D_E \in \clb(\he),
\]
and there exists $\Psi \in \hinf$ such that
\[
(H_\Phi D_E S_E)^*P_E = \Psi P_E,
\]
where $D_E$ denotes the differential operator on the space of $E$-valued polynomials.
\item $\M^{S_E}_{S_E}(T_\Phi)$ is similar to a contraction if and only if
\[
\Phi = 0.
\]
\item $\M_{S_E}^{S_E}(H_\Phi)$ is similar to a contraction if and only if
\[
\Phi \in (z-\bar{z})\linf + \zhinf .
\]
\item $\M^{S^*_E}_{S_E}(T_{\Phi})$ is similar to a contraction if and only if
\[
\Phi \in (\bar{z} - z)\linf.
\]
\end{enumerate}
\end{theorem}

These results are new, even for $E=\C$. However, we do not know how to apply the present methodology to the similarity problem for $\M^{S^*_E}_{S_E}(H_{\Phi})$, as considered by Pisier. Consequently, this case is not included in the above list nor in the list of the definitions of Foguel-type operators. The remaining two cases, namely, $\M^{S^*_E}_{S_E^*}(T_{\Phi})$ and $\M^{S^*_E}_{S_E^*}(H_{\Phi})$, follow from the above results and have been outlined in Section \ref{sec: T & H}.

We remark that, as will become evident in the proofs to follow, except for the cases of $\M^{S_E}_{S^*_E}(H_{\Phi})$ and $\M^{S_E}_{S_E}(T_{\Phi})$, all of the aforementioned results reveal a somewhat unexpected connection to the theory of Toeplitz + Hankel operators. This is an emerging area of independent interest and has been studied by several mathematicians (cf. \cite{NSJ, Deift, Ehr}). It is reasonable to suggest that the present article marks a step forward in supporting the applicability of Toeplitz + Hankel operators.

In the context of the problem of similarity of $\M^{S^*_E}_{S_E}(H_{\Phi})$ to contractions, we note that Aleksandrov and Peller \cite{Aleksandrov} resolved the case $E=\C$, while the work of Davidson and Paulsen \cite{Davidson-Paulsen} addresses the case when $\dim(E)<\infty$.

We apply some of our results to concrete cases, such as when $E = \C$. For instance, we prove that for the Hilbert-Hankel operator $H_\psi$, the corresponding Foguel-type operator $M^S_{S^*}(H_\psi)$ is never similar to a contraction. Recall that the matrix representation of the Hilbert-Hankel operator $H_\psi$ is given by:
\[
H_\psi = \left\lbrack \frac{1}{i+j+1}\right\rbrack_{i, j\geq 0}.
\]

It should be noted that there are already generic results in literature, which, in particular, also classify Foguel-type operators in terms of their similarity to contractions. For instance, \cite[Corollary 4.2]{Cassier} asserts that
\begin{equation}\label{Similar-intro-1}
\M^{S_E}_{S^*_E}(X)\, \mbox{is similar to contraction if and only if it is power bounded.}
\end{equation}

Also, we have the following general result, known as Foia\c{s}-Williams criterion (\cite {Foias-Williams}, also see \cite{CCFW} and see \cite[Corollary 5.15]{Paulsen} for a homological proof): Let $E$ and $E_*$ be Hilbert spaces, and let $Y\in\clb(E)$ and $Z\in\clb(E_*)$. Assume that $Y$ and $Z$ are similar to contractions, and either $Y$ is a coisometry or $Z$ is an isometry. Then $\M^Y_Z(X)$ on $E\oplus E_*$ is similar to a contraction if and only if there exists bounded linear operator $A: E_* \raro E$ such that
\begin{equation}\label{Similar-intro-2}
YA-AZ=X.
\end{equation}

A key step in proving Theorem \ref{thm: intr} (in particular, part (1) and part (2)) is an improvement of \eqref{Similar-intro-2} in the context of \eqref{Similar-intro-1}, which we also identify as one of our main results of this paper (see Theorem \ref{main-th1}): Given $X\in \clb(\he)$, the following are equivalent:
\begin{enumerate}
\item $\M^{S_E}_{S^*_E}(X)$ is similar to a contraction.
\item There exist $A \in \clb(\he)$ and $\Psi \in \hinf$ such that
\[
A^*P_E = \Psi P_E \text{ and } X = A - S_E^*AS_E^*.
\]
\item There exist $A \in \clb(\he)$ and $\Psi \in \hinf$ such that
\[
A P_E = X P_E,  A^* P_E = \Psi P_E, \text{ and } X S_E = A S_E - S_E^*A.
\]
\end{enumerate}

This result has been proved in the following section, which in fact generalizes Proposition 5.5 and Remark 5.6 from \cite{Cassier}, through a substantially different approach.

The remainder of the paper is organized as follows. In Section \ref{sec: TO and sim}, we present characterizations of Foguel-type operators associated with Toeplitz operators that are similar to contractions, while in Section \ref{sec: HO and sim}, we do the same for Hankel operators. In Section \ref{sec: T & H}, we complete the similarity problem in our context by considering all remaining cases of Foguel-type operators involving both Toeplitz and Hankel operators. The final section provides examples and some general remarks.

\section{A general criterion}\label{sec: gene crit}

The purpose of this section is to give a general characterization of operators of the form
\[
\M^{S_E}_{S^*_E}(X) = \begin{bmatrix}
	S_E & X\\
	0 & S^*_E
\end{bmatrix} : \he \oplus \he \to \he \oplus \he,
\]
for some $X \in \clb(H^2_E(\D))$ that are similar to contractions. It is known, in view of \eqref{Similar-intro-1}, that such an operator $M^{S_E}_{S^*_E}(X)$ is similar to a contraction if and only if it is power bounded. However, the answer in terms of power boundedness doesn't seem to give a satisfactory solution to the similarity problem. In this section, we address both of these characterization problems. We need a lemma.

Given the operator $M^{S_E}_{S^*_E}(X)$ as above, we set
\begin{equation}\label{s2-eq1}
X_n = S^{n-1}_EX +S^{n-2}_EXS^*_E + S^{n-3}_EXS^{*^2}_E+ \cdots +  XS^{*^{n-1}}_E
\end{equation}
for each $n \geq 1$.
The idea of this notion comes from the following computation:
\begin{equation}\label{eqn: X*n}
\left(\M^{S_E}_{S^*_E}(X)\right)^n = \begin{bmatrix}
S^n_E & X_n\\
0 & S^{*^n}_E
\end{bmatrix},
\end{equation}
for all $n \geq 1$. It is mentioned in \cite{Cassier-Timotin} that the power boundedness of $\M^{S_E}_{S^*_E}(X)$ is equivalent to the boundedness of the sequence $\{\|X_n\|\}_n$. Since no proof is readily available, we provide one below to prove the equivalence of (2) and (3) for the sake of completeness. The equivalence of (1) and (2) follows from \eqref{Similar-intro-1}.

\begin{lemma}\label{lma-pb}
Let $X \in \clb(H^2_E(\D))$. The following are equivalent:
\begin{enumerate}
\item $\M^{S_E}_{S^*_E}(X)$ is similar to contraction.
\item $\M^{S_E}_{S^*_E}(X)$ is power bounded.
\item $\sup_n \|X_n\| < \infty$.
\end{enumerate}
\end{lemma}
\begin{proof}
We need to prove that (2) and (3) are equivalent. Let $\M^{S_E}_{S^*_E}(X)$ be power bounded. Then there exists $M> 0$ such that
\[
\sup_n \left \|\left(\M^{S_E}_{S^*_E}(X)\right)^n \right \| \leq M.
\]
Fix $f \in \he$ and $n\geq 1$. We know, by \eqref{eqn: X*n}, that
\[
\left(\M^{S_E}_{S^*_E}(X)\right)^n = \begin{bmatrix}
S^n_E & X_n\\
0 & S^{*^n}_E
\end{bmatrix},
\]
for all $n \geq 1$. Then
\begin{align*}
\|X_nf\|^2&\leq\|X_nf\|^2+\|S_E^*f\|^2\\
&=\left \|\left(\M^{S_E}_{S^*_E}(X)\right)^n \begin{bmatrix}
0\\ f
\end{bmatrix}\right \|^2 \\
&\leq \left \|\left(\M^{S_E}_{S^*_E}(X)\right)^n \right \|^2 \left \|\begin{bmatrix}
0\\ f
\end{bmatrix}\right \|^2\\
& \leq M^2 \|f\|^2,
\end{align*}
that is, $\|X_n\| \leq M$. This proves that $\{\|X_n\| \}_n$ is bounded. Conversely, assume that $\sup_n \|X_n\| \leq M$ for some $M>0$. Set
\[
M_1 = \max\{1,M\}.
\]
Then for any $f_1$ and $f_2$ in $\he$, we have
\begin{align*}
\left \|\left(\M^{S_E}_{S^*_E}(X)\right)^n \begin{bmatrix}
f_1\\ f_2
\end{bmatrix}\right \|^2 &= \left \| \begin{bmatrix}
S_E^nf_1 +X_nf_2\\ S^{*^n}_E f_2
\end{bmatrix}\right \|^2\\
& = \|S_E^nf_1 +X_nf_2\|^2 + \|S^{*^n}_E f_2\|^2\\
&  \leq \left(\|S_E^nf_1\| + \|X_nf_2\|\right)^2 + \|S^{*^n}_E f_2\|^2\\
&  \leq \left(\|f_1\| + M\|f_2\|\right)^2 + \| f_2\|^2\\
&  \leq M_1^2\left(\|f_1\| +\|f_2\|\right)^2 + M_1^2(\| f_1\|^2 + \| f_2\|^2)\\
&  \leq 2M_1^2\left(\| f_1\|^2 + \| f_2\|^2\right) + M_1^2(\| f_1\|^2 + \| f_2\|^2)\\
& = 3M_1^2 \left \|\begin{bmatrix}
f_1\\ f_2
\end{bmatrix}\right \|^2.
\end{align*}
Consequently,
\[
\left \|\left(\M^{S_E}_{S^*_E}(X)\right)^n\right \|\leq \sqrt{3}M_1,
\]
and hence $\M^{S_E}_{S^*_E}(X)$ is power bounded. This completes the proof of the lemma.
\end{proof}

With the above lemma in hand, we are now ready to state and prove the main theorem of this section, which gives Foia\c{s}-Williams type criteria for operator $\M^{S_E}_{S_E^*}(X)$ being similar to a contraction.

\begin{theorem}\label{main-th1}
Let $X\in \clb(\he)$. Then the following are equivalent:
\begin{enumerate}
\item $\M^{S_E}_{S^*_E}(X)$ is similar to a contraction.
\item There exist $A \in \clb(\he)$ and $\Psi \in \hinf$ such that
\[
A^*P_E = \Psi P_E,
\]
and
\[
X = A - S_E^*AS_E^*.
\]
\item There exist $A \in \clb(\he)$ and $\Psi \in \hinf$ such that
\[
A P_E = X P_E \text{ and } A^* P_E = \Psi P_E,
\]
and
\[
X S_E = A S_E - S_E^*A.
\]
\end{enumerate}
\end{theorem}
\begin{proof}
We first prove that (2) and (3) are equivalent. Assuming (2) is true, there exist $A \in \clb(\he)$ and $\Psi \in \hinf$, such that $A^*P_E = \Psi P_E$, and
\[
X = A - S_E^*AS_E^*.
\]
The above condition immediately implies that
\[
XS_E = AS_E - S_E^*A.
\]
Again, $X = A - S_E^*AS_E^*$ gives
\[
XP_E = (A - S_E^*AS_E^*) P_E = A P_E - S_E^*AS_E^* P_E = A P_E,
\]
as $S_E^* P_E = 0$. Therefore, we have $XP_E = A P_E$. Conversely, if (3) is true, then there exist $A \in \clb(\he)$ and $\Psi \in \hinf$ such that $A P_E = X P_E$, $A^* P_E = \Psi P_E$, and
\[
XS_E = AS_E - S_E^*A.
\]
In view of $I = P_E+S_ES_E^*$, we compute
\begin{align*}
X&=X(P_E+S_ES_E^*)\\
&=AP_E+(AS_E-S_E^*A)S_E^*\\
&=A(P_E+S_ES_E^*)-S_E^*AS_E^*\\
&=A-S_E^*AS_E^*.
\end{align*}
This finishes the proof of the equivalence of (2) and (3). We now proceed to establish the equivalence between (1) and (2), which requires more involved computations. First, assume that (1) is true. In particular, $\M^{S_E}_{S^*_E}(X)$ is power bounded, and hence by Lemma \ref{lma-pb}, we have
\[
\sup_n \|X_n\| < \infty,
\]
where $X_n$, $n \geq 1$, is given as in \eqref{s2-eq1}:
\[
X_n = S^{n-1}_EX +S^{n-2}_EXS^*_E + S^{n-3}_EXS^{*^2}_E+ \cdots +  XS^{*^{n-1}}_E.
\]
For time being, fix a natural number $n$. Note that
\[
S_EX_n = S^{n}_EX +S^{n-1}_EXS^*_E + S^{n-2}_EXS^{*^2}_E+ \cdots +  S_EXS^{*^{n-1}}_E,
\]
and
\[
X_nS_E^* =S^{n-1}_EXS_E^* +S^{n-2}_EXS^{*^2}_E + S^{n-3}_EXS^{*^3}_E+ \cdots +  XS^{*^{n}}_E.
\]
By taking the difference of the above quantities, we have
\[
S_EX_n- X_nS_E^* =S^{n}_EX - XS^{*^{n}}_E.
\]
Multiplying by $S_E^{*n}$, this yields
\[
S^{*^{n-1}}_EX_n- S^{*^{n}}_EX_nS_E^* =X - S^{*^{n}}_EXS^{*^{n}}_E,
\]
that is,
\[
S^{*^{n-1}}_EX_{n}- S^*_E\left(S^{*^{n-1}}_EX_{n}\right)S_E^* =X - S^{*^{n}}_EXS^{*^{n}}_E.
\]
We want to apply a limiting argument to the above, or, more specifically, for a subsequence. We proceed as follows. Since $\{\|X_n\|\}_n$ is bounded, it follows that the sequence $\{\|S^{*^{n-1}}_E X_n\|\}_n$ is also bounded (note that $\|S_E^{*m}\| = 1$ for all $m \in \Z_+$). Furthermore, by the Banach–Alaoglu theorem, the closed unit ball of $\clb(\he)$ is compact in the weak operator topology. Therefore, there exists a subsequence, which we denote by $\{S_E^{*{n_j-1}} X_{n_j}\}$, that converges in the weak operator topology. Let
\[
WOT-\lim_n S^{*^{n_j-1}}_EX_{n_j} = A \in \clb(\he).
\]
This, together with the fact that ${S_E^*}^n\to 0$ as $n\to \infty$ in the strong operator topology, finally yields the desired limiting property:
\begin{equation}\label{s2-eq3}
A-S^*_EAS^*_E =X.
\end{equation}
Substituting the above in the expression of $X_n$, we get
\[
X_n = S^{n-1}_E(A-S^*_EAS^*_E) +S^{n-2}_E(A-S^*_EAS^*_E)S^*_E + \cdots + (A-S^*_EAS^*_E)S^{*^{n-1}}_E.
\]
As $S_E S_E^* = I - P_E$, we have
\[
X_n = S^{n-1}_E A - S^{n-2}_E(I-P_E)AS^*_E +S^{n-2}_EAS^*_E - S^{n-3}_E(I-P_E)AS^{*^2}_E + \cdots + AS^{*^{n-1}}_E-S^*_EAS^{*^{n}}_E.
\]
We write this as
\[
X_n = S^{n-1}_E A -S^*_EAS^{*^{n}}_E -S^{n-1}_EP_EA+ Q_n,
\]
where
\[
Q_n := S^{n-1}_EP_EA+S^{n-2}_EP_EAS^*_E+S^{n-3}_EP_EAS^{*^2}_E + \cdots +P_EA S^{*^{n-1}}_E.
\]
Since
\[
\|S^{n-1}_E A -S^*_EAS^{*^{n}}_E -S^{n-1}_EP_EA\| \leq 3 \|A\|,
\]
boundedness of $\{\|X_n\|\}_n$ ensures that $\{\|Q_n\|\}_n$ is also bounded. Consequently, the sequence $\{\|Q_n^*\|\}_n$ is bounded as well. Note that
\begin{equation}\label{s2-eq2}
Q_n^* = A^*P_E S^{*^{n-1}}_E +S_EA^*P_ES^{*^{n-2}}_E+S^{2}_EA^*P_E S^{*^{n-3}}_E+ \cdots +S^{{n-1}}_EA^* P_E.
\end{equation}
Suppose now $\|Q_n^*\|\leq C$, for some constant $C>0$. Fix $n \in \mathbb{N}$ and $e\in E$. By using the property of the backward shift $S_E^*$, we have
\begin{align*}
Q_n^*(z^ke) = \begin{cases}
z^{n-k-1}A^*e & \text{if $0\leq k\leq n-1$}\\
0 & \text{if $ k \geq n$}.
\end{cases}
\end{align*}
Given an $E$-valued polynomial
\[
\mathbf{p} = e_0+ ze_1+ z^2e_2+ \cdots + z^{n-1}e_{n-1} \in \he,
\]
with $e_i\in E$ for all $i=0, 1, \ldots, n-1$, define
\[
\tilde{\mathbf{p}} =e_{n-1} + ze_{n-2}+ + z^2e_{n-3}+ \cdots +z^{n-1}e_0.
\]
Therefore, we have
\begin{align*}
Q_n^*(\tilde{\mathbf{p}}) &= Q_n^*(e_{n-1}) + Q_n^*(ze_{n-2}) + Q_n^*(z^2e_{n-3})+ \cdots + Q_n^*(z^{n-1}e_0)\\
& = z^{n-1}A^*e_{n-1} + z^{n-2}A^*e_{n-2} +z^{n-3}A^*e_{n-3}+ \cdots + A^*e_0.
\end{align*}
Define
\[
R(\mathbf{p}):=\sum_{i=0}^{n-1}z^iA^*e_i,
\]
for $E$-valued polynomial $\mathbf{p}$ as above. It is now evident that
\begin{equation}\label{inq-1}
\|R(\mathbf{p})\|=\|Q_n^*(\tilde{\mathbf{p}})\|\leq C \|\tilde{\mathbf{p}}\|= C\|{\mathbf{p}}\|.
\end{equation}
This means $R$ is a well-defined bounded linear operator from the space of $E$-valued polynomials to $\he$. Since $E$-valued polynomials form a dense subspace of $\he$, the operator $R$ admits a unique bounded linear extension to an element of $\clb(\he)$, which we also call $R$. Furthermore, it is easy to see that \[
S_ER(\mathbf{p}) =RS_E(\mathbf{p}),
\]
for each $E$-valued polynomial $\mathbf{p}\in\he$. Now pick $f\in \he$. Choose a sequence of $E$-valued polynomials $\{\mathbf{q}_n\}\subseteq \he$ such that
\[
\lim_n \mathbf{q}_n = f,
\]
in $\he$. By a standard limiting argument, we have
\[
S_ER(f) = S_E \lim_{n\to \infty} R(\mathbf{q}_n) =\lim_{n\to \infty}S_E R(\mathbf{q}_n) = \lim_{n\to \infty}RS_E(\mathbf{q}_n) = RS_E(f),
\]
that is, $S_ER = RS_E$. Therefore, $R$ is an analytic Toeplitz operator, that is, there exists an analytic symbol $\Psi \in \hinf$ such that (cf. \cite[Corollary 3.20]{Radj-Rosen})
\[
R=T_\Psi.
\]
From the definition of $R$, we see that
\[
A^*e = R e = T_\Psi e =\Psi e,
\]
for all $e\in E$. This fact combined with (\ref{s2-eq3}) ascertains the validity of (2). On the other hand, suppose (2) is true. Then there exist $A \in \clb(\he)$ and $\Psi \in \hinf$ such that
\[
X = A - S_E^*AS_E^*,
\]
and $A^* P_E = T_\Psi P_E$. To show that (1) is true under this assumption, it is enough to show that
\[
\sup_n \|X_n\| < \infty,
\]
where $X_n$'s are given as in \eqref{s2-eq1} (see Lemma \ref{lma-pb}). Using the equality $X = A - S_E^*AS_E^*$, we can now perform exactly the same lines of computation succeeding (\ref{s2-eq3}), to obtain an expression for $X_n$ as
\[
X_n = S^{n-1}_E A -S^*_EAS^{*^{n}}_E -S^{n-1}_EP_EA+ Q_n.
\]
As a consequence, showing that $\{\|X_n\|\}_n$ is bounded amounts to showing $\{\|Q_n^*\|\}_n$ (as already pointed out earlier, where $Q_n^*$ is defined in (\ref{s2-eq2})) is bounded. Fix $f= \sum_{i=0}^{\infty}z^ie_i \in \he$, and observe that
\begin{align*}
\|Q_n^*(f)\|&=\left\|\sum_{i=0}^{n-1}Q_n^*(z^ie_i)+\sum_{i=n}^\infty Q_n^*(z^ie_i)\right\|
\\
&= \|z^{n-1}A^*e_0 +z^{n-2}A^*e_1 +\cdots + A^*e_{n-1}\|
\\
&= \|z^{n-1}\Psi e_0 +z^{n-2}\Psi e_1 +\cdots + \Psi e_{n-1}\|
\\
&=\|T_\Psi(z^{n-1}e_0 +z^{n-2}e_1 +\cdots +e_{n-1})\|
\\
& \leq \|T_\Psi\|\|z^{n-1}e_0 +z^{n-2}e_1 +\cdots + e_{n-1}\|
\\
& \leq\|\Psi\|_{\infty} \|f\|,
\end{align*}
which completes the proof.
\end{proof}

The above theorem provides a somewhat satisfactory answer to the question of the similarity of $\M^{S_E}_{S^*_E}(X)$ to contractions in general. However, one might expect a more concrete answer when the operator $X$ is more natural and tied with the shift operator, particularly in the cases of Toeplitz and Hankel operators. This will be the theme of the following two sections.

\section{Toeplitz operators}\label{sec: TO and sim}

The purpose of this section is to present a concrete characterization of symbols $\Phi \in \linf$ such that $\M^{S_E}_{S_E^*}(T_\Phi)$ is similar to a contraction. On this occasion, we will examine the usefulness of the equivalence of the conditions in Theorem \ref{main-th1} in the context of Toeplitz operators.

We need to recall some results concerning Toeplitz + Hankel operators. Let $T, H \in \clb(\he)$. Then $T$ is a Toeplitz operator if and only if
\[
S_E^* T S_E = T.
\]
And, $H$ is a Hankel operator if and only if
\[
S_E^* H = H S_E.
\]
An operator $X \in \clb(\he)$ is said to be \textit{Toeplitz + Hankel operator} if there exist symbols $\Phi, \Omega \in \linf$ such that
\[
X = T_\Phi + H_\Omega.
\]
In this case, we simply say that $X$ is a Toeplitz + Hankel operator. In what follows, we will use the following characterization of Toeplitz + Hankel operators established in \cite[Theorem 2.2]{NSJ}:

\begin{theorem}\label{thm: T + H}
Let $A \in \clb(\he)$. The following conditions are equivalent:
\begin{enumerate}
\item $A$ is a Toeplitz + Hankel operator.
\item $A S_E - S_E^*A$ is a Toeplitz operator.
\item $S_E^*A S_E - A$ is a Hankel operator.
\item $A S_E + S_E^{*2} A S_E = S_E^* A + S_E^* A S_E^2$.
\end{enumerate}
\end{theorem}

Recall that
\[
\zhinf = \{\Phi \in \linf: \Phi = \sum_{n=1}^{\infty}\bar{z}^n\Phi_n, \Phi_n \in \clb(E), n \geq 1\}.
\]
Recall also that for each $\Psi \in \linf$, we define $\tilde\Psi \in \linf$ by
\[
\widetilde{\Psi}(z) = \Psi(\bar{z}),
\]
for $z \in \T$ almost everywhere. Moreover, by the definition of Toeplitz operators, we have the following relation:
\[
T_\Phi P_E = P_{\he} M_\Phi|_{\he} P_E = P_{\he} M_\Phi P_E,
\]
which, according to our convention, is written more concisely as
\[
T_\Phi P_E = P_{\he}\Phi P_E.
\]
In other words, for each $e \in E$, we have
\[
(T_\Phi P_E)(e) = (P_{\he}\Phi P_E)(e) = P_{\he}(\Phi e).
\]
With this in hand, we now proceed to the main result of this section.

\begin{theorem}\label{Cor-th1}
Let $\Phi \in \linf$. Then $\M^{S_E}_{S^*_E}(T_{\Phi})$ is similar to a contraction if and only if
\[
\Phi \in (1-\bar{z}^2)\linf,
\]
and
\[
\Phi^* + \widetilde{\Phi}^* \in (1-z^2)(\hinf +\zhinf).
\]
\end{theorem}
\begin{proof}
To start with, let us assume $\M^{S_E}_{S^*_E}(T_{\Phi})$ to be similar to a contraction. By part (3) of Theorem \ref{main-th1}, there exist $A \in \clb(\he)$ and $\Psi \in \hinf$ such that
\[
T_\Phi S_E = AS_E - S_E^*A,
\]
with $A P_E = T_\Phi P_E$ and $A^* P_E = \Psi P_E$. Since $S_E = T_z\in\clb(H^2_E(\D))$ and $z$ is analytic, we have $T_\Phi S_E = T_{z\Phi}$, and hence, the first identity yields
\[
T_{z\Phi} = AS_E - S_E^*A.
\]
By virtue of Theorem \ref{thm: T + H}, $A$ is a Toeplitz + Hankel operator, that is, there exist symbols $\Theta, \Omega \in \linf$ such that
\[
A= T_{\Theta} + H_{\Omega}.
\]
Substituting this value of $A$  into the identity $T_{z\Phi} = AS_E - S_E^*A$ and then using the Hankel property $S_E^* H_\Omega = H_\Omega S_E$, we obtain:
\begin{align*}
T_{z\Phi} & = (T_{\Theta} + H_{\Omega})S_E - S_E^*(T_{\Theta} + H_{\Omega})
\\
& = T_{\Theta} S_E - S_E^*T_{\Theta}
\\
& = T_{(z-\bar{z})\Theta}.
\end{align*}
In the above, we have used the relations $T_{\Theta} S_E = T_{z\Theta}$ and $S_E^*T_{\Theta} = T_{\bar{z}\Theta}$ (recall again that $S_E = T_z$). Therefore, by the uniqueness of symbol for Toeplitz operators, we have
\[
z\Phi = (z-\bar{z})\Theta,
\]
which implies that $\Phi = (1 - \bar{z}^2)\Theta\in (1-\bar{z}^2)\linf$, and hence
\[
\Theta = \frac{\Phi}{1-\bar{z}^2}.
\]
Since $A P_E = T_\Phi P_E$, we further obtain that
\begin{align*}
P_{\he} \Phi P_E & = P_{\he} M_\Phi P_E
\\
& = T_{\Phi} P_E
\\
&= A P_E
\\
& = (T_{\Theta} + H_{\Omega})P_E
\\
& = P_{\he} \Theta P_E + P_{\he} \Omega P_E
\\
& = P_{\he}\left(\frac{\Phi}{1-\bar{z}^2} P_E \right) + P_{\he}\left(\Omega P_E\right).
\end{align*}
Here we have used the fact that $H_{\Omega} P_E = P_{\he} M_\Omega J P_E = P_{\he} M_\Omega P_E$. From the above, we have
\[
P_{\he}\left(\left(\Phi - \frac{1}{1-\bar{z}^2}\Phi \right)P_E\right) = P_{\he}\left(\Omega P_E\right),
\]
and hence
\[
P_{\he}\left(\frac{\bar{z}^2}{\bar{z}^2-1}\Phi P_E \right) = P_{\he}\left(\Omega P_E\right).
\]
Set
\[
\Upsilon := \Omega -\frac{\bar{z}^2}{\bar{z}^2-1}\Phi \in \linf.
\]
Then we have that $P_{\he}\left(\Upsilon P_E\right) =0$. This implies
\[
\Upsilon \in \zhinf.
\]
Note that ${H^*_\Omega} = H_{\widetilde\Omega^*}$, and
\[
M_{\Theta}^* = M_{\frac{\Phi^*}{1-z^2}}.
\]
Also, rewrite $\Omega$ as
\[
\Omega = \Upsilon + \frac{\bar{z}^2}{\bar{z}^2-1}\Phi.
\]
Then from the condition $\Psi P_E = A^* P_E$, we get
\begin{align*}
\Psi P_E &=(T_\Theta^*+H_\Omega^*)P_E
\\
&=	P_{\he}\left(M_\Theta^* P_E\right) + P_{\he}\left(M_{\widetilde\Omega^*} P_E\right)
\\
& = P_{\he}\left(\frac{1}{1-z^2} \Phi^* P_E\right) + P_{\he}\left(\widetilde\Upsilon^* P_E + \frac{\bar{z}^2}{\bar{z}^2-1} \widetilde{\Phi}^* P_E\right)
\\
& = P_{\he}\left(\frac{1}{1-z^2} \Phi^* P_E\right) + P_{\he}\left(  \frac{\bar{z}^2}{\bar{z}^2-1} \widetilde{\Phi}^* P_E \right).
\end{align*}
If we set
\[
\Lambda = \frac{1}{1-z^2} \Phi^* + \frac{\bar{z}^2}{\bar{z}^2-1} \widetilde{\Phi}^* -\Psi \in \linf,
\]
then the above identity converts to
\[
P_{\he}\left( \Lambda P_E \right)= 0,
\]
which implies $\Lambda \in \zhinf$, and consequently
\[
\frac{1}{1-z^2} \Phi^* + \frac{1}{1-z^2} \widetilde\Phi^* = \Lambda + \Psi.
\]
This gives us the desired condition that
\[
\Phi^* + \widetilde \Phi^* \in (1-z^2)(\hinf +\zhinf).
\]
Conversely, assume that $\Phi \in (1-\bar{z}^2)\linf$ and
\[
\Phi^* + \widetilde \Phi^* \in (1-z^2)(\hinf +\zhinf).
\]
There exist $\Theta \in \linf$, $\Psi \in \hinf$, and $\Upsilon \in \zhinf$ such that
\[
\Phi = (1-\bar{z}^2)\Theta,
\]
and
\[
\Phi^* + \widetilde \Phi^* = (1-z^2)(\Psi +\Upsilon).
\]
Define
\[
\Omega = \frac{\bar{z}^2}{\bar{z}^2-1}\Phi \in \linf.
\]
Note that $\Theta + \Omega = \Phi$. Set
\[
A = T_\Theta + H_\Omega.
\]
Using again the property of Hankel operators $H_\Omega S_E = S_E^* H_\Omega$, we have
\begin{align*}
AS_E - S_E^*A  & = (T_{\Theta} + H_{\Omega})S_E - S_E^*(T_{\Theta} + H_{\Omega})\\
& = T_{\Theta} S_E - S_E^*T_{\Theta}\\
& = T_{(z-\bar{z})\Theta} \\
& = T_{z\Phi}= T_\Phi S_E.
\end{align*}
Furthermore,
\[
A P_E = (T_{\Theta} + H_{\Omega}) P_E= P_{\he}\left((\Theta + \Omega)P_E\right)=P_{\he}\left(\Phi P_E\right) = T_\Phi P_E,
\]
that is, $A P_E = T_\Phi P_E$. Finally, since
\begin{align*}
A^* P_E & = (T_{\Theta} + H_{\Omega})^*P_E
\\
&= P_{\he}\left((\Theta^* + \widetilde \Omega^*) P_E \right)
\\
& = P_{\he}\left(\frac{1}{1-z^2} (\Phi^* + \widetilde\Phi^*) P_E \right)
\\
& = P_{\he}\left((\Psi +\Upsilon)P_E\right)
\\
& = P_{\he} \Psi P_E
\\
& = \Psi P_E,
\end{align*}
by Theorem \ref{main-th1}, we conclude that $\M^{S_E}_{S^*_E}(T_{\Phi})$ is similar to a contraction. Our proof is therefore complete.
\end{proof}

In a later section, we will see an application of the $E=\C$ case of this result.

%further apply this result to the case where $E$ is the complex plane.

\section{Hankel operators}\label{sec: HO and sim}

In this section, we will characterize symbols $\Phi$ that make the Foguel-type operator $\M^{S_E}_{S_E^*}(H_\Phi)$ similar to a contraction. The goal is, of course, to verify the applicability of Theorem \ref{main-th1}, as we did for Toeplitz operators in the previous section. We need the concept of differential operators.

For each $n \in \Z_+$ and $e \in E$, define
\[
D_E (z^n e) = \begin{cases}
nz^{n-1} e & \mbox{if } n \geq 1 \\
0 & \mbox{if } n = 0.
\end{cases}
\]
The \textit{differential operator} $D_E$ on the space of $E$-valued polynomials extends the action of $D_E$ as defined on $E$-valued monomials. Therefore, for all $\mathbf{p}=\sum_{n=0}^{M}z^ne_n\in \he$, we have
\[
D_E(\mathbf{p})=\mathbf{p}^\prime,
\]
where $\mathbf{p}^\prime = \sum_{n=1}^{M}nz^{n-1}e_n$. It is worth clarifying that, by $A\in\clb(\he)$, where $A$ is originally defined on polynomials, we mean that $A$ admits a unique extension to a bounded linear operator on the entire space $\he$. With all these preparations, we are ready for characterizations of Foguel-type operators $\M^{S_E}_{S^*_E}(H_{\Phi})$ that are similar to contractions:

\begin{theorem}\label{Cor-th2}
Let $\Phi \in \linf$. Then $\M^{S_E}_{S^*_E}(H_{\Phi})$ is similar to a contraction if and only if
\[
H_\Phi D_E \in \clb(\he),
\]
and there exists $\Psi \in \hinf$ such that
\[
(H_\Phi D_E S_E)^*P_E = \Psi P_E.
\]
\end{theorem}
\begin{proof}
Assume that $\M^{S_E}_{S^*_E}(H_{\Phi})$ is similar to a contraction. Then by Theorem \ref{main-th1}, there exist $A \in \clb(\he)$ and $\Psi \in \hinf$ such that
\[
H_{\Phi}S_E = AS_E - S_E^*A,
\]
with $A P_E = H_{\Phi} P_E$ and $A^* P_E = \Psi P_E$. As $D_E S_E P_E = P_E$, it follows that
\[
A P_E = H_\Phi P_E = H_\Phi D_E S_E P_E.
\]
To apply mathematical induction, let us fix a natural number $m$ and suppose that
\[
A(z^m e) = H_\Phi D_ES_E(z^m e),
\]
for all $e\in E$. Fix $e \in E$. Again the condition $H_{\Phi}S_E = AS_E - S_E^*A$ gives
\begin{align*}
H_{\Phi}(z^{m+1}e) & = H_{\Phi} S_E (z^{m}e)
\\
& = (AS_E - S_E^*A)(z^{m}e)
\\
& = A(z^{m+1} e)- S_E^*A(z^m e)
\\
& = A(z^{m+1} e)- S_E^*H_\Phi D_ES_E(z^m e),
\end{align*}
as $A(z^m e) = H_\Phi D_ES_E(z^m e)$. Using the property of Hankel operators, namely $H_\Phi S_E = S_E^* H_\Phi$, it follows that
\begin{align*}
H_{\Phi}(z^{m+1}e) & = A(z^{m+1} e)- H_\Phi S_E D_E(z^{m+1} e) = A(z^{m+1} e)- (m+1)H_{\Phi}(z^{m+1}e),
\end{align*}
and hence
\[
A(z^{m+1} e) =(m+2)H_{\Phi}(z^{m+1}e)= H_\Phi D_E S_E(z^{m+1} e).
\]
By the principle of mathematical induction, we finally conclude that
\[
A(z^n e) = H_\Phi D_ES_E(z^n e),
\]
for all $n\geq 1$ and $e\in E$. Therefore,
\[
A(\mathbf{p})= H_\Phi D_ES_E(\mathbf{p}),
\]
for all $E$-valued polynomials $\mathbf{p}\in \he$. Since $A\in\clb(\he)$, the linear operator $H_\Phi D_ES_E$ between the space of $E$-valued polynomials and $\he$ is bounded. Recall that these polynomials constitute a dense subspace of $\he$, and thus $H_\Phi D_ES_E$ admits a unique bounded linear extension to an element of $\clb(\he)$, which we also call
\[
H_\Phi D_E S_E.
\]
In other words,
\[
A= H_\Phi D_E S_E,
\]
on all of $\he$, and hence, the condition $A^* P_E = \Psi P_E$ yields
\[
(H_\Phi D_E S_E)^* P_E = \Psi P_E.
\]
We now need to show that $H_\Phi D_E$, defined originally on $E$-valued polynomials, can also be extended uniquely to an element in $\clb(\he)$. To see this, for an $E$-valued polynomial $\mathbf{p}\in \he$, we compute
\[
\begin{split}
(H_\Phi D_ES_E)S_E^*(\mathbf{p}) & = H_\Phi D_ES_ES_E^*(\mathbf{p}) = H_\Phi D_E(I-P_E)(\mathbf{p}) = H_\Phi D_E(\mathbf{p}).
\end{split}
\]
Since both $H_\Phi D_E S_E$ and $S_E^*$ are bounded on $\he$, it follows that $H_\Phi D_E$ extends to a bounded linear operator on $\he$.

\noindent Conversely, assume that $H_\Phi D_E \in\clb(\he)$ and there is a function $\Psi \in \hinf$ such that
\[
(H_\Phi D_E S_E)^* P_E = \Psi P_E.
\]
Then we set
\[
A:= H_\Phi D_E S_E \in \clb(\he).
\]
Clearly,
\[
A^* P_E = (H_\Phi D_E S_E)^* P_E = \Psi P_E.
\]
On the other hand, since $D_E S_E e = e$ for all $e \in E$, it follows that
\[
A P_E = H_\Phi D_E S_E P_E = H_\Phi P_E.
\]
Finally, for an $E$-valued polynomial $\mathbf{p}\in \he$, we compute
\begin{align*}
AS_E(\mathbf{p}) - S_E^*A(\mathbf{p})  &= (H_\Phi D_E S_E)(z\mathbf{p}) - S_E^*(H_\Phi D_E S_E)(\mathbf{p})
\\
& = H_\Phi(2z\mathbf{p}+z^2\mathbf{p}^\prime)- S_E^*H_\Phi(\mathbf{p}+z\mathbf{p}^\prime)
\\
& = H_\Phi(2z\mathbf{p} + z^2\mathbf{p}^\prime) - H_\Phi S_E(\mathbf{p}+z\mathbf{p}^\prime)
\\
& = H_\Phi(z\mathbf{p})
\\
& = H_\Phi S_E(\mathbf{p}).
\end{align*}
Using the boundedness of $A$ and the density of $E$-valued polynomials in $\he$, we conclude that
\[
AS_E(f) - S_E^*A(f)=H_\Phi S_E(f)
\]
for all $f\in\he$. By Theorem \ref{main-th1}, we now conclude that $\M^{S_E}_{S^*_E}(H_{\Phi})$ is similar to a contraction. This completes our proof.
\end{proof}

We will later discuss the $E=\C$ case of the above result in the form of Theorem \ref{s6-th1}, as well as some of its applications.

Let us also remark that differential operators in the context of the similarity problem for operators of the form $\M_{S_E}^{S^*_E}(X)$ have appeared in \cite{Davidson-Paulsen} as well.

\section{Toeplitz and Hankel operators}\label{sec: T & H}

In the previous two sections, we have answered the similarity problem for $\M^{S_E}_{S_E^*}(X)$, where $X$ is a Toeplitz or Hankel operator. In the present section, we provide complete answers for the similarity problem for $M^Y_Z(X)$ when
\[
Y = Z = S_E,
\]
or
\[
Y = Z = S_E^*,
\]
and $X$ is a Toeplitz or a Hankel operator. We will also completely analyze the remaining case $\M_{S_E}^{S_E^*}(X)$, where $X$ is a Toeplitz operator.

We begin by highlighting a reduction result for the first four classes of operators. By the Foia\c{s}-Williams criterion \eqref{Similar-intro-2}, we have the following fact:
\[
\M^{S_E}_{S_E}(X) = \begin{bmatrix}
S_E & X\\
0 & S_E
\end{bmatrix} : \he \oplus \he \to \he \oplus \he
\]
is similar to a contraction if and only if there exists $A\in\clb(\he)$ such that
\begin{equation}\label{s3-eq1}
S_E A-AS_E=X.
\end{equation}
Similarly,
\[
\M^{S_E^*}_{S_E^*}(X) = \begin{bmatrix}
S_E^* & X\\
0 & S_E^*
\end{bmatrix} : \he \oplus \he \to \he \oplus \he
\]
is similar to a contraction if and only if there exists $A^\prime\in\clb(\he)$ such that
\begin{equation}\label{s3-eq2}
S_E^* A^\prime-A^\prime S_E^*=X.
\end{equation}
Clearly, $A^\prime$ satisfies \eqref{s3-eq2} if and only if $A={-A^\prime}^*$ satisfies \eqref{s3-eq1}, with $X$ replaced by $X^*$. Hence, $\M^{S_E^*}_{S_E^*}(X)$ is similar to contraction if and only if $\M^{S_E}_{S_E}(X^*)$ is. Now, considering our choices $X=T_\Phi$ and $H_\Phi$, we note that both $T_\Phi^*$ and $H_\Phi^*$ are again Toeplitz and Hankel operators, respectively.

Therefore, it suffices to focus on the operator $\M^{S_E}_{S_E}(X)$, where $X$ is a Toeplitz or a Hankel operator. We show below that the first class of operators is trivial.

\begin{theorem}\label{s3-th1}
Let $\Phi \in \linf$. Then $\M^{S_E}_{S_E}(T_\Phi)$ is similar to a contraction if and only if
\[
\Phi = 0.
\]
\end{theorem}
\begin{proof}
If $\Phi = 0$, then $\M^{S_E}_{S_E}(T_\Phi)$ is itself a contraction. Let us now assume that $\M^{S_E}_{S_E}(T_\Phi)$ is similar to a contraction. By \eqref{s3-eq1}, there exists $A\in \clb(\he)$ such that
\[
S_E A -A S_E = T_\Phi.
\]
Multiplying both sides on the left by $S_E^*$, we obtain
\[
A - S_E^* A S_E = T_{\bar{z}\Phi}.
\]
By the property of Toeplitz operators, we know
\[
{S_E^*}^{n-1} T_{\bar{z}\Phi} S_E^{n-1} = T_{\bar{z}\Phi},
\]
which in turn gives
\[
\begin{split}
{S_E^*}^{n-1}A S_E^{n-1} - S_E^{*n} A S_E^n & = {S_E^*}^{n-1}(A -  S_E^* A S_E) S_E^{n-1}
\\
& = {S_E^*}^{n-1} T_{\bar{z}\Phi} S_E^{n-1}
\\
& = T_{\bar{z}\Phi},
\end{split}
\]
for all $n\geq 1$. Fix $n \geq 1$ and write
\[
A - S_E^{*^n}A S_E^n =\sum_{k=1}^{n}\left(S_E^{*^{k-1}}A S_E^{k-1} -S_E^{*^k}A S_E^k\right).
\]
It now follows that
\[
A - S_E^{*^n}A S_E^n = nT_{\bar{z}\Phi},
\]
and hence
\[
n\|\Phi\|_{\infty} = n\|T_{\bar{z}\Phi}\| =\|A - S_E^{*^n}A S_E^n\| \leq 2\|A\|,
\]
for all $n\geq 1$. Consequently, $\Phi = 0$.
\end{proof}

In view of the above result and the observation in the beginning of this section, we have the following: Let $\Phi \in \linf$. Then $\M^{S_E^*}_{S_E^*}(T_\Phi)$ is similar to a contraction if and only if
\[
\Phi = 0.
\]

Before proceeding, we need to establish some basic identities relating Toeplitz and Hankel operators. For the purpose of completeness, we prefer to provide full proofs of these identities, even though these are well known in the scalar case and probably familiar to specialists in the general vector-valued case. Recall that, for each $\Phi \in \linf$, we define $\widetilde \Phi \in \linf$ by
\[
\widetilde{\Phi}(z) = \Phi(\bar{z}).
\]

\begin{lemma}\label{TH-id}
Let $\Phi, \Psi \in \linf$. The following two identities hold:
\begin{enumerate}
\item $T_{\Phi\Psi}- T_\Phi T_\Psi = H_{\bar{z}\Phi}H_{\bar{z}\widetilde{\Psi}}$.
\item $H_{\bar{z}\widetilde{\Phi} \widetilde{\Psi}} = H_{\bar{z} \widetilde{\Phi}}T_\Psi + T_{\widetilde{\Phi}}H_{\bar{z}\widetilde{\Psi}}. $
\end{enumerate}
\end{lemma}
\begin{proof}
Let us first note that
\[
JP_{\he}J = M_z(I-P_{\he})M_{\bar{z}},
\]
which can be easily verified on the standard orthonormal basis of $\Le$, and hence is true for any arbitrary element in $\Le$. Using this, we have for any $f\in \he$:
\begin{align*}
T_{\Phi\Psi}(f)- T_\Phi T_\Psi(f) & = P_{\he}(\Phi \Psi f)-P_{\he}(\Phi P_{\he}(\Psi f))
\\
& = P_{\he}\left(\Phi(I-P_{\he})\Psi f\right)
\\
& = P_{\he}\left(\Phi M_{\bar{z}} JP_{\he}J M_z\Psi f\right)
\\
& = P_{\he}\left(\bar{z}\Phi JP_{\he}\left(\bar{z}\Psi(\bar{z})Jf\right)\right)
\\
& = P_{\he}\left(\bar{z}\Phi J H_{\bar{z}\Psi(\bar{z})}(f)\right)
\\
& = H_{\bar{z}\Phi}H_{\bar{z}\widetilde{\Psi}}(f).
\end{align*}
This proves the first identity. For the second identity, we compute, for each $f\in \he$,
\begin{align*}
H_{\bar{z}\widetilde{\Phi} \widetilde{\Psi}} (f)- H_{\bar{z}\widetilde{\Phi}} T_\Psi(f) & = P_{\he}\left(\bar{z} \widetilde{\Phi}\widetilde{\Psi} Jf \right) - P_{\he}\left(\bar{z} \widetilde{\Phi} JP_{\he}(\Psi f)\right)
\\
& = P_{\he}\left(\bar{z} \widetilde{\Phi} J(\Psi f)\right) - P_{\he}\left(\bar{z} \widetilde{\Phi} JP_{\he}(\Psi f)\right)
\\
& = P_{\he}\left(\bar{z} \widetilde{\Phi} J(I-P_{\he})(\Psi f)\right)
\\
& = P_{\he}\left(\bar{z} \widetilde{\Phi} J M_{\bar{z}} JP_{\he}J M_z(\Psi f)\right)
\\
&= P_{\he}\left(\bar{z} \widetilde{\Phi} M_zJ^2P_{\he}(\bar{z} \widetilde{\Psi} Jf)\right)
\\
&= P_{\he}\left(\widetilde{\Phi} P_{\he}(\bar{z} \widetilde{\Psi} Jf)\right)
\\
& = P_{\he}\left(\widetilde{\Phi} H_{\bar{z} \widetilde{\Psi} }(f)\right)
\\
& = T_{\widetilde{\Phi}}H_{\bar{z} \widetilde{\Psi}}(f).
\end{align*}
This completes the proof of the lemma.
\end{proof}

With the above information in hand, we proceed to solve the similarity to contraction problem for $\M_{S_E}^{S_E}(H_\Phi)$.

\begin{theorem}\label{s3-th2}
Let $\Phi \in \linf$. Then $\M_{S_E}^{S_E}(H_\Phi)$ is similar to a contraction if and only if
\[
\Phi \in (z-\bar{z})\linf + \zhinf .
\]
\end{theorem}
\begin{proof}
Let $\M_{S_E}^{S_E}(H_\Phi)$ be similar to a contraction. In view of $(\ref{s3-eq1})$, there exists $A\in \clb(\he)$ such that
\begin{equation}\label{h-eq-1}
S_E A -A S_E = H_\Phi,
\end{equation}
and hence
\[
A -S_E^*A S_E = S_E^*H_\Phi=H_{\bar{z}\Phi}.
\]
By Theorem \ref{thm: T + H}, we know that $A$ is a Toeplitz + Hankel operator, that is, there exist symbols $\Theta, \Omega \in \linf$ such that
\[
A = T_{\Theta} + H_{\Omega}.
\]
Substituting this value of $A$ in $A -S_E^*A S_E=H_{\bar{z}\Phi}$ gives
\[
H_\Omega - S_E^* H_\Omega S_E = H_{\bar{z}\Phi}.
\]
Applying the Hankel property $S_E^* H_\Omega = H_\Omega S_E$ to the above, we obtain
\[
H_{(1-\bar{z}^2)\Omega} = H_{\bar{z}\Phi},
\]
which implies
\[
(1-\bar{z}^2)\Omega = \bar{z}\Phi + \widehat{\Psi},
\]
for some
\[
\widehat{\Psi} = \sum_{n=1}^{\infty}\bar{z}^n\Psi_{n-1} \in \zhinf.
\]
This yields
\[
(z-\bar{z})\Omega = \Phi + z\widehat{\Psi},
\]
which can be rewritten as
\[
(z-\bar{z})\Omega = \Phi + (z-\bar{z})\widehat{\Psi}+\bar{z}\widehat{\Psi}.
\]
As a result,
\[
\Phi = (z-\bar{z}) (\Omega-\widehat{\Psi}) - \bar{z}\widehat{\Psi}\in (z-\bar{z})\linf + \zhinf.
\]
Conversely, let $\Phi \in (z-\bar{z})\linf + \zhinf$. In other words, there exist $\Omega \in \linf$ and $\Psi\in \zhinf$ such that
\[
\Phi = (z-\bar{z}) \Omega + \Psi.
\]
Set
\[
\Theta (z) = -\bar{z}^2 \widetilde\Omega.
\]
Clearly, $\Theta \in\linf$. By Lemma \ref{TH-id}, we have
\[
T_{z\Theta}- T_{z}T_\Theta = H_{I_E}H_{\bar{z} \widetilde\Theta},
\]
and
\[
T_{z} H_\Omega = H_{z\Omega} - H_{I_E}T_{\bar{z} \widetilde\Omega}.
\]
Here $I_E$ denotes the identity operator in $\clb(E)$. Using the above information, and the fact that $H_{I_E}=P_E$, we compute
\begin{align*}
S_E(T_{\Theta} + H_{\Omega}) -(T_{\Theta} + H_{\Omega}) S_E&= (T_{z}T_\Theta -T_{z\Theta}) + (T_{z}H_\Omega -H_\Omega T_{z})
\\
& = - H_{I_E}H_{\bar{z}\widetilde\Theta} + H_{z\Omega} -  H_{I_E}T_{\bar{z} \widetilde\Omega} -H_{\bar{z}\Omega}
\\
& = - H_{I_E}H_{\bar{z} \widetilde\Theta} +  H_{(z-\bar{z})\Omega} -  H_{I_E}T_{\bar{z} \widetilde\Omega}
\\
& =  P_EH_{z\Omega } +  H_{\Phi} -  P_ET_{\bar{z} \widetilde\Omega},
\end{align*}
that is,
\begin{equation}\label{s3-eq4}
S_E(T_{\Theta} + H_{\Omega}) -(T_{\Theta} + H_{\Omega}) S_E = H_\Phi + \left(P_EH_{z\Omega }-  P_ET_{\bar{z} \widetilde\Omega} \right).
\end{equation}
For each $e\in E$, we now have
\[
H_{z\Omega }^*(e)= H_{z \widetilde\Omega^*}(e) = P_{\he}(z \widetilde\Omega^*e),
\]
and also
\[
T^*_{\bar{z}\widetilde{\Omega}}(e)=T_{z\Omega(\bar{z})^*}(e)= P_{\he}(z \widetilde\Omega^*e).
\]
As a result, we have $H_{z\Omega }^*P_E = T^*_{\bar{z} \widetilde\Omega}P_E$, or, equivalently,
\[
P_EH_{z\Omega }=  P_ET_{\bar{z} \widetilde\Omega}.
\]
Hence we have from (\ref{s3-eq4}):
\[
S_E(T_{\Theta} + H_{\Omega}) -(T_{\Theta} + H_{\Omega}) S_E =  H_\Phi,
\]
that is, $A= T_{\Theta} + H_{\Omega}\in\clb(\he)$ is a solution of the operator equation \eqref{h-eq-1}. By the Foia\c{s}-Williams criterion (see \eqref{s3-eq1} again), we conclude that $\M_{S_E}^{S_E}(H_\Phi)$ is similar to a contraction. This completes the proof.
\end{proof}

In view of the discussion following \eqref{s3-eq2} and Theorem \ref{s3-th2}, we note that $\M^{S_E^*}_{S_E^*}(H_\Phi)$ is similar to a contraction if and only if
\[
\widetilde{\Phi}^* \in (z-\bar{z})\linf + \zhinf,
\]
or, equivalently
\[
\Phi \in (z-\bar{z})\linf + \zhinf.
\]
In other words, $\M^{S_E}_{S_E}(H_\Phi)$ is similar to a contraction if and only if $\M^{S_E^*}_{S_E^*}(H_\Phi)$ is similar to a contraction as well.

We conclude this section by addressing the remaining case of the similarity problem, namely, the similarity of $\M^{S_E^*}_{S_E}(X)$ to a contraction, where $X$ is a Toeplitz operator. Note that in this setting, the Foia\c{s}-Williams criterion  (see (\ref{Similar-intro-2})) for $\M^{S_E^*}_{S_E}(X)$ to be similar to a contraction becomes equivalent to the existence of an operator $A\in\clb(\he)$ such that
\[
S_E^* A-AS_E=X.
\]
Here we again use the characterization of Toeplitz + Hankel operators as we have done earlier in several places.

\begin{theorem}\label{s4-Th1}
Let $\Phi \in \linf$. Then $\M^{S^*_E}_{S_E}(T_{\Phi})$ is similar to a contraction if and only if
\[
\Phi \in (\bar{z} - z)\linf.
\]
\end{theorem}
\begin{proof}
In the case where $X = T_\Phi$, the Foia\c{s}-Williams criterion, as discussed above, states that $\M^{S^*_E}_{S_E}(T_{\Phi})$ is similar to a contraction if and only if there exists $A\in \clb(\he)$ such that
\[
S_E^*A -A S_E = T_\Phi.
\]
Now, assume that $\M^{S^*_E}_{S_E}(T_{\Phi})$ is similar to a contraction. The characterization of Toeplitz + Hankel operators, as stated in Theorem \ref{thm: T + H}, then guarantees the existence of symbols $\Theta$ and $\Omega$ in $\linf$ such that
\[
A = T_\Theta + H_\Omega.
\]
As $ S_E^* H_\Omega = H_\Omega S_E$, it follows that
\[
T_\Phi = S_E^*(T_\Theta + H_\Omega) -(T_\Theta + H_\Omega) S_E = T_{(\bar{z}-z)\Theta},
\]
which gives
\[
\Phi = (\bar{z}-z)\Theta \in (\bar{z} - z)\linf.
\]
Conversely, assume that there exists a symbol $\Theta \in \linf$ such that $\Phi = (\bar{z}-z)\Theta$. Then it is easy to see that
\[
S_E^*T_\Theta -T_\Theta S_E = T_\Phi.
\]
Again the Foia\c{s}-Williams criterion, as pointed out above, guarantees that $\M^{S^*_E}_{S_E}(T_{\Phi})$ is similar to a contraction, completing the proof of this theorem.
\end{proof}

We remark that, as a consequence of a general theorem, it was observed in \cite[Section 5.1]{Cassier} that an operator of the form $M^{S_E}_{S_E}(X)$ is similar to a contraction if and only if it is similar to an isometry.

\section{Examples}

In this concluding section, we present some concrete examples, focusing on the scalar case, of the results obtained so far. We also provide some general remarks. In the case where $E = \mathbb{C}$, we will omit the subscripts in the notation for the Hardy space, bounded analytic functions, and related objects. For instance, we will simply write
\[
H^2_{\mathbb{C}}(\mathbb{D}) = H^2(\mathbb{D}),
\]
and also
\[
S_{\mathbb{C}} = S.
\]
We begin with a simple application of Theorem \ref{Cor-th1}. Pick a function $\theta\in H^\infty(\mathbb{D})$, and set
\[
\vp(z) = (1-\overline{z}^2)\theta(\overline{z}),
\]
for $z \in \T$-a.e. Clearly, $\vp\in L^\infty(\mathbb{T})$. It is evident that
\[
\vp\in(1-\overline{z}^2)L^\infty(\mathbb{T}).
\]
For $z \in \T$-a.e. we also compute
\begin{align*}
({\vp}^* + \widetilde{\vp}^*)(z) & = (1-z^2)\overline{\theta(\overline{z})}+(1-\overline{z}^2)\overline{\theta(z)}\\
&=(1-z^2)\left(\overline{\theta(\overline{z})}-\overline{z}^2\overline{\theta(z)}\right),
\end{align*}
which implies
\[
{\vp}^* + \widetilde{\vp}^* \in (1-z^2)\left(H^\infty(\mathbb{D})+\overline{zH^\infty(\mathbb{D})}\right).
\]
Theorem \ref{Cor-th1} then implies that ${M}^{S}_{S^*}(T_\vp)$ is similar to a contraction.

\subsection{BMOA functions} We will now focus on Hankel operators $H_\vp\in\clb(H^2(\D))$, $\vp\in L^\infty(\T)$. To this end, let us first introduce the BMOA functions (cf. \cite{Peller book}). Denote by $m$ the normalized Lebesgue measure on $\T$. Let $I$ be a subarc of $\T$, and let $f$ be an absolutely integrable function defined on $\T$. Define
$$
f_I =\frac{1}{m(I)}\int_I f\, dm.
$$
Then $f$ is called a function of bounded mean oscillation, if
$$
\sup_I\frac{1}{m(I)}\int_I|f-f_I|\,dm<\infty .
$$
The space of such functions is named as \emph{BMO}. The analytic subspace of BMO of interest is
\[
BMOA= BMO \cap H^2(\D),
\]
whose elements are known as bounded mean oscillation analytic functions. It is well known that
\[
BMOA=\{P_{H^2(\D)} (\vp): \vp\in L^\infty(\T)\}.
\]
Let us now pick $\vp\in L^\infty(\mathbb{T})$, and define
\[
f:=P_{H^2(\D)} (\vp).
\]
%From the above information, it is clear that $H_\vp\in\clb({H^2(\D))}$ if and only if $f\in\mbox{BMOA}$.
By virtue of the Theorem 4.4, Corollary 4.5 and 4.7 of \cite{Davidson-Paulsen}, it is immediately seen that the membership of $f^\prime$ in BMOA is equivalent to the fact that
\[
H_\vp D\in\mathcal{B}(H^2(\mathbb{D})).
\]
Consider now the Fourier expansion of $f$ on $\T$ as
\[
f(z)=\sum_{n \in \Z_+} a_nz^n.
\]
For each $k  \in \Z_+$, we compute
\begin{align*}
\left\langle (H_\vp DS)^*(1), z^k\right\rangle & = \left\langle1, (H_\vp DS)(z^k)\right\rangle
\\
&=\left\langle 1, (k+1)\sum_{n=k}^\infty a_nz^{n-k}\right\rangle
\\
&=(k+1)a_k.
\end{align*}
This implies
\[
(H_\vp DS)^*(1)=\sum_{n=0}^\infty (n+1)a_nz^n=(zf)^\prime.
\]
In view of the above discussion, we immediately see that Theorem \ref{Cor-th2} takes the following form for $E=\C$:

\begin{theorem}\label{s6-th1}
Let $\vp\in L^\infty(\T)$. Then $M^{S}_{S^*}(H_\vp)$ is similar to a contraction if and only if
\[
\left(P_{H^2(\D)}(\vp)\right)^\prime\in \mbox{BMOA},
\]
and
\[
\left(zP_{H^2(\D)}(\vp)\right)^\prime\in H^\infty(\mathbb{D}).
\]
\end{theorem}
In particular, we can choose $\vp$ to be a complex polynomial, or, for each $|\alpha|>1$, we can choose $\vp$ as
\[
\vp(z) = \frac{1}{\alpha-z},
\]
for all $z\in\T$. For such choices of symbols $\vp$, we can conclude by Theorem \ref{s6-th1} that $M^{S}_{S^*}(H_\vp)$ is similar to a contraction.

\subsection{Hilbert-Hankel matrix} Consider the Hilbert-Hankel matrix
\[
H = \begin{bmatrix}
1 & \frac{1}{2} & \frac{1}{3} & \cdots \\
\\
\frac{1}{2} & \frac{1}{3} & \frac{1}{4} & \cdots \\
\\
\frac{1}{3} & \frac{1}{4} & \frac{1}{5} & \cdots \\
\cdot & \cdot & \cdot & \ddots
\end{bmatrix}.
\]
This is a linear operator defined on $l^2(\Z_+)$, and its boundedness was established by Hilbert. This classical matrix represents the Hankel operator corresponding to the symbol $\psi \in L^\infty(\mathbb{T})$ defined by
\[
\psi(e^{it}) =ie^{-it}(\pi -t),
\]
for all $t\in [0,2\pi)$. More specifically, we observe that
\[
P_{H^2(\mathbb{D})}(\psi) = \sum_{n=0}^\infty\frac{z^n}{n+1}\in\mbox{BMOA}.
\]
However,
$$
\left(P_{H^2(\mathbb{D})}(\psi)\right)^\prime=\sum_{n=1}^\infty\frac{n}{n+1}z^{n-1}\notin H^2(\D),
$$

As a consequence of Theorem \ref{s6-th1}, we have the following result:

\begin{corollary}\label{cor: Hilbert Hankel}
${M}^{S}_{S^*}(H_\psi)$ is not similar to a contraction.
\end{corollary}

The same conclusion holds for the operator $M_{S}^{S^*}(H_\psi)$, since it follows from \cite[Theorems 4.1 and 4.2]{Aleksandrov} that the operator $M_{S}^{S^*}(H_\psi)$ is similar to a contraction if and only if
\[
\left(P_{H^2(\mathbb{D})}(\psi)\right)^\prime \in \text{BMOA}.
\]
However, in the present case, we have already seen that
\[
\left(P_{H^2(\mathbb{D})}(\psi)\right)^\prime \notin \text{BMOA}.
\]

\subsection{General remarks} Addressing the similarity problem for Foguel-type operators is a natural and significant question in the theory of Hilbert function spaces. Moreover, we remark that Paulsen’s identification of the ``similarity to a contraction'' problem with the notion of complete polynomial boundedness \cite{Paulsen1} highlights the inherent difficulty of characterizing operators similar to contractions. Indeed, complete polynomial boundedness requires verification at all matrix levels, as it involves the contractivity of operator-valued polynomial matrices of arbitrary size. From all these perspectives, we feel that the results reported here, all concrete in nature, will be relevant for different purposes as well.

Let us conclude this paper with a curious observation regarding Theorems \ref{s3-th2} and \ref{s4-Th1}: Given $\Phi\in\linf$, if $M^{S^*_E}_{S_E}(T_\Phi)$ is similar to a contraction, then so is $M_{S_E}^{S_E}(H_\Phi)$.

\bigskip

\noindent\textbf{Acknowledgement:} The first named author is supported by a National Postdoctoral Fellowship (N-PDF) provided by the Anusandhan National Research Foundation, India (File number: PDF/2025/000089). The research of the second named author is supported by a post-doctoral fellowship provided by the National Board for Higher Mathematics (NBHM), India (Order No: 0204/16(8)/2024/R\&D- II/6760, dated May 09, 2024). The research of the third named author is supported in part by ANRF, Department of Science \& Technology (DST), Government of India (File No: ANRF/ARGM/2025/000130/MTR).

\end{document}